\documentclass[a4paper,12pt]{amsart}
\usepackage{CJK}
\usepackage{romannum}
\usepackage{amsfonts}
\usepackage{mathtools}
\usepackage{amsmath,amscd}

\usepackage{ifthen}
\usepackage{amsrefs}
\usepackage{mathrsfs}
\usepackage{amsthm}
\usepackage{amssymb}

\usepackage{tikz-cd}
  \usepackage{graphicx}
 \usepackage{relsize}

\usepackage{MnSymbol}
\usetikzlibrary{positioning, decorations.text}

\usepackage{hyperref}
\usepackage{url}
\usepackage{etoolbox}
\usepackage{rotating}

\usepackage[shortlabels]{enumitem}
\usepackage[paper=a4paper,left=20mm,right=20mm,top=25mm,bottom=30mm]{geometry}

\usepackage{tikz}

\usetikzlibrary{backgrounds}
\usepackage{tkz-euclide}
\usepackage{xcolor}
\usetikzlibrary{trees,snakes,shapes.geometric}
\usepackage[outline]{contour}
\contourlength{1.5pt}
\usetikzlibrary{positioning}
\usetikzlibrary{%
  matrix,%
  calc,%
  arrows%
}

\setlist[enumerate]{topsep=0em, itemsep= -0em, parsep = 0 em, label=$(\alph*)$}

\nocite{*}

\newcommand{\cH}{\mathcal{H}}
\newcommand{\cL}{\mathcal{L}}
\newcommand{\cK}{\mathcal{K}}
\newcommand{\cM}{\mathcal{M}}

\newcommand{\cV}{\mathcal{V}}

\newcommand{\revddots}{\reflectbox{$\ddots$}}

\DeclareMathOperator{\rep}{rep}

\DeclareMathOperator{\rk}{rk}

\DeclareMathOperator{\modd}{mod}

\DeclareMathOperator{\dimu}{\underline{dim}}

\newtheorem{proposition}{Proposition}[section]
\newtheorem{Theorem}[proposition]{Theorem}
\newtheorem{Lemma}[proposition]{Lemma}

\newtheorem{corollary}[proposition]{Corollary}

\newenvironment{example}[1][Example.]{\begin{trivlist}
\item[\hskip \labelsep {\bfseries #1}]}{\end{trivlist}}

\newenvironment{Definition}[1][Definition.]{\begin{trivlist}
\item[\hskip \labelsep {\bfseries #1}]}{\end{trivlist}}

\author{Jie Liu}
\address{School of Mathematics and Statistics, Guangdong University of Technology,
Guangzhou 510520, People’s Republic of China}

\email{jie@gdut.edu.cn}

\address{Shenzhen International Center for Mathematics,   Southern University of Science and Technology, shenzhen 518055, People’s Republic of China }

\title{ elementary modules of the
generalized Kronecker quivers }

\begin{document}

\rmfamily


\thispagestyle{empty}

\maketitle

\begin{abstract}
Let $k$ be an algebraically closed field. The generalized or $n$-Kronecker quiver $K(n)$ is the  quiver with two vertices, called a source and a sink, and $n$ arrows from source to sink.  We use  $\modd\cK_n$ to denote the category of the finite-dimensional modules of the path algebra  $\cK_n=kK(n)$.  There exist  a function called dimension vector  $\dimu: \modd \cK_n\rightarrow \mathbb{Z}^2; M \mapsto (\dim_k M_1, \dim_k M_2 )$, and two  functors $\sigma: \mathbb{Z}^2 \rightarrow \mathbb{Z}^2; (x,y)\mapsto (nx-y,x)$, and $\delta: \mathbb{Z}^2\rightarrow \mathbb{Z}^2;(x,y)\mapsto (y,x)$.  Then the set  $\mathbf{F}=\{(x,y)\mid \frac{2}{n}x\leq y\leq x\}$ is the  fundamental domain of the dimension vectors of  regular modules  under the action of   $\delta$ and $\sigma$ in  $\modd\cK_n$.  We call a regular module $M\in \modd\cK_n$ \textit{elementary} if there does not exist a short exact sequence $(0)\rightarrow L\rightarrow M\rightarrow N\rightarrow (0)$ with $L,N$ being non-zero regular.  In this note, we focus on the set $\mathbf{F}$,  and when $x<n$, we establish a one-to-one correspondence between  elementary modules and  spaces of fixed rank.
 
\end{abstract}

\providecommand{\keywords}[1]
{
  \small	
  \textbf{\text{Keywords:}} #1
}


\section{introduction}

 Let $k$ be an algebraically closed  field, and let $Q$ be a finite, connected and wild quiver without  oriented cycles. If we use  $\rep_k(Q)$ to  denote the finite-dimensional representations of $Q$,  then there  exists an equivalence between categories  $\rep_k(Q)$ and  $\modd kQ$, where $kQ$ is the path algebra of quiver $Q$. Hence  we will frequently identify these two subcategories.

 Since the representation type of $Q$ is wild \cite[1.3]{Kerner}, it is hopeless to classify all the indecomposable modules. Hence it is desirable to find the modules with certain property in the category $\modd kQ$. Let $\tau$ be the Auslander--Reiten translation on $\modd kQ$, and let $M\in \modd kQ$ be an indecomposable module. We  say that $M$ is \textit{regular}, provided $\tau^t M\neq (0)$ for all $t\in \mathbb{Z}$.   Suppose that $M$ is regular.  Then  $M$ is said to be \textit{elementary} (or $\dimu M$ is \textit{elementary}) if there is no short exact sequence $(0)\rightarrow L\rightarrow M\rightarrow N\rightarrow (0)$ with $L,N\in\modd kQ$ being non-zero regular
 modules. Clearly, a module $M$ is elementary if and only if $\tau M$
 is elementary. For any module $M$ there exists a filtration 
 \begin{center}
   $(0)=M_0\subseteq M_1 \subseteq M_2\subseteq \cdots \subseteq M_p=M$
 \end{center}
such that $M_i/M_{i-1}$ is an elementary module, $1 \leq i\leq p$. Hence we can take elementary modules as simple objects in the full subcategory of $\modd kQ$ consisting of regular modules.  Unfortunately, for the wild quiver $Q$, not much is known about its elementary modules.

 Since the  generalized  Kronecker quiver $K(n)$
 \[
\begin{tikzcd}
    1 \circ
  \arrow[r, draw=none, "{\vdots}" description]
    \arrow[r, bend left,        "\gamma_1"]
    \arrow[r, bend right, swap, "\gamma_n"]
    &
    \circ 2 
\end{tikzcd}
\]
is particularly of interest, it has been studied by some authors (cf.\cite{Claus2}, \cite{Daniel2}), where $n\geq 3$. Depending on their discussion, we want to locate the elementary modules in $\modd\cK_n$. Using the Euler--Ringel form   $q(x,y)=x^2+y^2-nxy$  on the dimension vectors in $\modd \cK_n$, we say that  a dimension vector  $(x,y)$ is  \textit{regular}  if $q(x,y)<0$. Let $\mathbf{R}$ be the set of regular dimension vectors. By abusing notation, we introduce two maps  $\sigma, \delta$ on the set $\mathbf{R}$,  where $\sigma(x,y)=(nx-y,x)$ and $\delta(x,y)=(y,x)$ for all $ (x,y)\in \mathbf{R}$.  In particular, these two maps induce two functors (we still use the same symbols) $\sigma, \delta:\modd \cK_n\rightarrow \modd \cK_n$ such that a regular module $M$ is elementary if and only if  $\sigma (M)$ or $\delta(M)$ is elementary module. Otto Kerner and Frank Lukas have shown that there are only finitely many   $(\sigma)^2$-orbits of dimension vectors of elementary modules for $\mathcal{K}_n$ (cf. \cite{Otto}),  it is possible to classify these  dimension vectors.

According to \cite{Daniel2}, the set $\mathbf{F}=\{(x,y)\mid \frac{2}{n}x\leq y\leq x\}$ is the fundamental domain of the set $\mathbf{R}$  under the action of $\sigma$ and $\delta$. Hence we only need to focus on the set $\mathbf{F}$.  Let $ (x,y)\in \mathbf{F}$ be an elementary dimension vector. Claus Michael Ringel classified the  elementary modules for $K(3)$ in 2016 (cf. \cite{Claus2}), and  the author  showed that $x<2n$ for $K(n)$ \cite[Lemma 3.6]{jie}. Recently, Daniel Bissinger proved that $y<n$ \cite[Proposition 3.4]{Daniel2}.  Combining with  these facts, we describe the elementary module when $n<x<2n$, and we construct  a class of elementary modules $\{X(x,y;x+y-1)\mid  (x,y)\in \mathbf{F}\}$ when $x<n$.

Let  $V^i$ denote the $i$-dimensional vector space over the field $k$, and let $L(V^j,V^i)$ be the linear space consisting of all $k$-linear transformations from  $V^j$ to  $V^i$. We follow Roy Westwick's notation, and  define $l(r,j,i)=\dim_k \cH$, where  $\cH  \subseteq L(V^j,V^i)$ is  a maximum linear subspace such
that  $\rk v=r$ for all $0\neq v\in \cH$. Such subspaces $\cH$ have been studied for many years (cf. \cite{West}, \cite{john}, \cite{de}), and it is our first time to connect them with certain modules of Kronecker quivers. We want to find the relationship between subspaces $\cH$ and elementary modules. Let   $(x,y)\in \mathbf{F}$ with $x<n$. We use $\cM(x,y)$ to denote the set of elementary modules with dimension vector $(x,y)$ for $K(x+y-1)$, and use $\cL(x,y)$ to denote the set of maximum linear subspaces $\cH$ in $L(V^y, V^{x+y-1})$ satisfying $\rk v=y$ for any $0\neq v\in\cH$.  Then we have

\begin{Theorem}
Let   $(x,y)\in \mathbf{F}$ with $x<n$. Then there exists a bijection map between the set $\cM(x,y)$ and the set $\cL(x,y)$ under the isomorphisms.

\end{Theorem}

 \section{preliminaries}
 
 In this section, we present a few concepts and some  background information.  For  convenience, we will give some definitions in a short way. A thorough introduction to this part can be found in  \cite[\Romannum{2}-\Romannum{7}]{Assem1}.
 
 A \textit{quiver} is just an oriented graph (loops and multiple arrows are allowed), usually denoted by  $Q=(Q_0, Q_1, s, t)$, where $s,t: Q_1\rightarrow Q_0$. The elements in $Q_0$ are called \textit{points} or \textit{vertices}, and the elements in $Q_1$ are called \textit{arrows}, respectively. A finite-dimensional representation $M=((M_x)_{x\in Q_0}, (M(\alpha))_{\alpha\in Q_1})$ over $Q$ consists of vector spaces $M_{x}$ and $k$-linear maps $M(\alpha): M_{s(\alpha)} \rightarrow M_{t(\alpha)}$ such that $\dim_{k}M:=\sum _{x\in Q_0}\dim_k M_{x}$ is finite. A \textit{morphism} $f: M\rightarrow N$ between two representations is a collection of $k$-linear maps $(f_z)_{z\in Q_0}$ such that for each arrow $\alpha: x\rightarrow y$ there is a commutative diagram

\begin{center}
$\begin{array}[c]{ccc}
M({x})&\stackrel{M(\alpha)}{\longrightarrow}&M(y)\\
\downarrow\scriptstyle{f_x}&&\downarrow\scriptstyle{f_y}\\
N(x)&\stackrel{N(\alpha)}{\longrightarrow}&N(y).
\end{array}$
\end{center}
We thus define a category $\rep_k(Q)$  of  finite-dimensional  representations  of $Q$.  
Normally, we will use  $S(i)$ to denote the simple representation and $P(i)$ (resp. $I(i)$) to denote the projective  (resp. injective) representation at the vertex $i, i\in Q_0.$    
 
In this note, we focus on the  $n$-Kronecker quiver $K(n)$ and we always assume that $n\geq 3$.   For the dimension vector $\dimu$ on $\modd \cK_n$, if  $(0)\rightarrow L \rightarrow M \rightarrow N\rightarrow (0)$ is an exact sequence in $\modd \cK_n$, then
\begin{center}
 $\dimu L+\dimu N= \dimu M$.
\end{center}
We denote by $<-,->$    the  bilinear form
 \begin{center}
$<-,->: \mathbb{Z}^2 \times \mathbb{Z}^2 \rightarrow \mathbb{Z}, ((x_1,x_2),(y_1,y_2))\mapsto (x_1y_1+x_2y_2)-nx_1y_2.$
\end{center}
This bilinear form  coincides with the Euler--Ringel form on the Grothendieck group $K_0(\cK_n)\cong \mathbb{Z}^2$. Then we denote the corresponding quadratic form by
\begin{center}
$q:\mathbb{Z}^2\rightarrow \mathbb{Z},x\mapsto <x,x>$.

\end{center}

\begin{Definition}
A dimension vector $(x,y)$ for $\cK_n$ is said to be regular,  provided $q(x,y)<0$. 

\end{Definition}

Let $\sigma, \sigma^-$ be the Bernstein-Gelfand-Ponomarev reflections (or BGP-functors, ) of $K_0(\cK_n)$ given by $\sigma(x,y)=(nx-y,x), \sigma^-(x,y)=(y,ny-x)$. Moreover,  we still use   $\sigma, \sigma^-$ to denote the BGP functors on $\modd \cK_n$ (we  take  the opposite of the $n$-Kronecker quiver to be again the $n$-Kronecker quiver).  We sometimes call such $\sigma, \sigma^-$ the \textit{shift functors} of category $\modd\cK_n$. Let $M\in\modd \cK_n$ be an indecomposable module.  Then

\begin{enumerate}
\item[(1)]    $M$ is said to be  \textit{preinjective},  provided there exists  $t\in \mathbb{N}_0$ such that  $\sigma^{-t} M =(0)$.

\item[(2)]  $M$ is said to be \textit{preprojective},  provided there exists  $t\in \mathbb{N}_0$ such that $\sigma^{t} M=(0)$. 

\item[(3)] A module is said to be \textit{regular} if it is neither   preprojective nor preinjective.

\end{enumerate}
Actually, we have $\tau=\sigma^2$ (cf. \cite{Gabriel}). If $M\in\modd\cK_n$ is an indecomposable module different from $S(2)$, then  $\dimu \sigma M=\sigma\dimu M$; similarly, if $M$ is indecomposable and different from $S(1)$, then we have $\dimu \sigma^-M=\sigma^- \dimu M$. When module $M$ is an elementary module, the module $\sigma^t M$  is also an elementary module for all $t\in \mathbb{Z}$ \cite[VII. Corollary 5.7$(d)$]{Assem1}. Equivalently, a module is elementary if and only if its submodule is preprojective or its factor module is preinjective \cite[Proposition 3.2]{Daniel2}.

\begin{Definition}
The dimension vector $(x,y)$ is said to be \textit{elementary} (\textit{preprojective}, or \textit{preinjective}),  provided there exists an  elementary  (preprojective, or preinjective)  module $M$ with $\dimu M=(x,y)$.
\end{Definition}

We now consider  the set   of regular dimension vectors  $\mathbf{R}=\{(x,y)\in \mathbb{N}^2\mid q(x,y)<0\}$. We have seen that $\sigma$ maps $\mathbf{R}$ onto $\mathbf{R}$. In fact, there is another transformation $\delta$ on $K_0(\cK_n)$ defined by $\delta(x,y)=(y,x)$, and  it also sends $\mathbf{R}$ onto $\mathbf{R}$.  Let $M\in \modd \cK_n$.  Then $\delta (\dimu M)=\dimu M^*$,  where $M^*$ is the dual representation of $M$, that is,  $M^*=(M^*_1,M^*_2,(M^*(\gamma_i))_{1\leq i\leq n})$, $M^*_1$  is the $k$-dual of $M_2$ and $M^*_2$ is the $k$-dual of $M_1$, the map $M^*(\gamma_i)$ is the $k$-dual of $M(\gamma_i)$. We put

\begin{center}
$\mathbf{F}=\{(x,y)\mid \frac{2}{n}x\leq y\leq x\}.$
\end{center}

Then we have 
\begin{Lemma}\cite[Section 2. Lemma]{Claus2}
The set $\mathbf{F}$ is a fundamental domain for the action of the group generated by $\delta$ and $\sigma$ on the set $\mathbf{R}.$ 
\end{Lemma}

For the elementary modules, we have 

\begin{Lemma} \cite[Lemma 3.1]{Claus2}\label{regular}
Assume that $M\in\modd \mathcal{K}_n$ is a regular module with a proper non-zero submodule $U$ such that both dimension vectors $\dimu U$ and $\dimu M/ U$ are regular. Then $M$ is not elementary.
\end{Lemma}

\section{Dimension vectors of elementary modules }

By duality and Lemma \ref{regular},  a module $M\in$ mod $\cK_n$ is elementary if and only if its dual $M^*$ is elementary. That is, if $\dimu M=(x,y)$, then $(x,y)$ is elementary if and only if $(y,x)$ is elementary. Hence we only need to study one of these two dimension vectors. Furthermore

\begin{Lemma}\cite[Lemma 14.11]{Daniel}\label{less n}
Let $M\in \modd \cK_n$  be an elementary module with $\dimu M=(x,y)$ and $y\leq x\leq y+n-2.$ Then $x< n.$
\end{Lemma}

\begin{corollary}\label{coro}\cite[Corollary 3.11]{jie}
Let $(x,y)\in \mathbf{F}$ with $x<n$. Suppose that $(x,y)$ is elementary for $\cK_{n}$. Then $(x,y)$ is elementary for $\cK_{n+1}$.

\end{corollary}

For the dimension of the linear subspace of fixed rank $\cH$ we have the following

\begin{Theorem} \cite[Theorem]{West}\label{wt}
Let $2\leq r\leq j\leq i$ be integers. Then 

\begin{center}
$i-r+1\leq l(r,j,i)\leq i+j-2r+1$.
\end{center}

\end{Theorem}

Let $A$ be a matrix over $k$. We use $A^t$ to denote its transpose. When $r=j$,  we consider the subspace $\cH$  with $\dim_k\cH=l(j,j,i)=i-j+1$ (cf. \cite[Corollary I]{john}).  Given the matrix 

\[  A_a=\begin{bmatrix}
a_1 &a_2  & a_3&   \cdots & a_{i-j+1} & 0&  0 & \cdots & 0\\
0 & a_1 & a_2 & \cdots & a_{i-j} & a_{i-j+1} & 0 &\cdots & 0\\
0& 0& a_1 & \cdots & a_{i-j-1} & a_{i-j} & a_{i-j+1} & \cdots & 0\\
\vdots & \vdots & \vdots & \ddots & \ddots & \vdots & \vdots & \ddots & \vdots\\
0 & 0 & 0  & \cdots & a_1 & a_2 & a_3 & \cdots & a_{i-j+1} 
\end{bmatrix}^t_{j\times i}, \]
it can be seen that such matrices could form a basis of $\cH$, where $a=(a_1,\cdots, a_{i-j+1})\in k^{i-j+1}$. Normally, the  subspace $\cH$ with $\dim_k \cH=i-j+1$ is not unique. For example, we can let 
 \[ C_a=\begin{bmatrix}
 0 & \cdots & 0  & 0 & a_{i-j+1} & \cdots & a_3 & a_2 & a_1\\
  0 & \cdots & 0   & a_{i-j+1}  & a_{i-j} & \cdots & a_2 & a_1 & 0\\
0 & \cdots &  a_{i-j+1} &  a_{i-j}& a_{i-j-1} &\cdots  &a_1 & 0  & 0\\
\vdots & \revddots&\revddots & \revddots& \revddots &\revddots & \vdots & \vdots & \vdots\\
a_{i-j+1} & \cdots & a_3 & a_2 & a_1 & \cdots & 0& \cdots & 0
\end{bmatrix}^t_{j\times i}, \]
and  all such matrices $C_a$ also consist of a basis of some largest subspace $\cH'$. Clearly, matrices $A_a$ and $C_a$ are linearly independent in general.

Let $\Lambda_n$ be the  space spanned by the arrows of $K(n)$. Then  it is an $n$-dimensional vector space with basis $\{\gamma_i\mid 1\leq i\leq n\}$. Let $M\in\modd \cK_n$, and let $m\in M$.  We define $\gamma_i.m:=M(\gamma_i)(m)$,  and let $\Lambda_n.m$ be the  submodule  of $M$ generated by the element $m$. Let $(x,y)\in \mathbf{F}$.  In \cite{jie}, we show that if $(x,y)$ is elementary, then $x<2n$ and $y<2(n-1)$ \cite[Lemma 3.3, 3.4]{jie}. Moreover, we have the following

\begin{Lemma}\label{2n} \cite[Lemma 3.8]{jie}
Suppose that  $(x,y)\in \mathbf{F}$ is an elementary dimension vector, where $n < x<2n$. Then 
 $y^2-y+2x\leq 4n+2$.

\end{Lemma}
In particular, Lemma \ref{2n} has an easy corollary comparing to the proof of \cite[Proposition 3.4]{Daniel2}.

\begin{corollary}\label{y}
Suppose that  $(x,y)\in \mathbf{F}$ is an elementary dimension vector. Then $y< n$.

\end{corollary}
\begin{proof}
  According to Lemma \ref{less n} and  \cite[Lemma 3.6]{jie},  we only need to focus on the case: $n<x<2n$. 
By $y^2-y+2x\leq 4n+2$, we have $y^2-y \leq 4n+2-2x \leq 4n+2-2(n+1)=2n$. Then $(y-\frac{1}{2})^2\leq 2n+\frac{1}{4}$, that is, $y\leq \frac{1+\sqrt{8n+1}}{2}$. However,  $\sqrt{8n+1}\leq 2n-1$, that is, $4n^2-12n=4n(n-3) \geq 0$ since we have $n\geq 3$. On the other hand, Claus Micheal Ringel tells us that $y<3$ when $n=3$ \cite[Section 1. Theorem]{Claus2}. Hence we always have $y<n$.

\end{proof}

When the dimension vector $(x,y)\in \mathbf{F}$ with $x<n$, we have 

\begin{Lemma}\cite[Lemma 3.10]{jie}\label{n-1}
Let $M\in\modd \mathcal{K}_n$ be a module with $\dimu M=(x,y)\in \mathbf{F}$, and let $0\neq m\in M_1$ be an arbitrary element, $x<n$.   Then $M$ is elementary if and only if the submodule $U$ generated by $m$ has dimension vector $(1,y)$.

\end{Lemma}

Let $ M \in  \modd\cK_n $ be a module. We introduce an arrow $ \gamma_{n+1} $ different from $ \gamma_i, 1 \leq i \leq n $, and we define $ \gamma_{n+1}.m = 0 $ for any $ m \in M $. Then $ \Lambda_{n+1} = \Lambda_n \oplus k \gamma_{n+1} $, and it induces naturally an embedding $ \iota : \modd\mathcal{K}_n \to \modd \mathcal{K}_{n+1} $. Hence the module $ M $ can be seen as an object located in $ \modd \mathcal{K}_{n+1} $. In the following, we always mean $ \gamma_{n+1}.M = 0 $ when we have $ M \in \modd \mathcal{K}_n \subseteq \modd  \mathcal{K}_{n+1} $. Let $ (x,y) \in \mathbf{F} $ with $ x < n $. We have $ \frac{2}{n} x \leq y \leq x $. Hence $ \frac{2}{n+1} x < \frac{2}{n} x \leq y \leq x $. Then $ (x,y) $ is located in the fundamental domain $ \mathbf{F}' $ of $ \mathcal{K}_{n+1} $, where $ \mathbf{F}' = \{(x,y) \mid \frac{2}{n+1} x \leq y \leq x\} $.

 Given  $(x,y)\in \mathbf{F}$ with $n<x<2n$,  we have

\begin{proposition}\label{N}
Let $(x,y)\in \mathbf{F}$, and let $M\in\modd \cK_n$ be a module with dimension vector $(x,y)$, where $n<x<2n$. Suppose that $M$ is an elementary module.  Then  there exist two elements  $ m', m''\in M_1$ such that  $\dimu \Lambda_n.m'=(1,y-1)$ and $\dimu \Lambda_n.m''=(1,y)$.
\end{proposition}

\begin{proof}
Lemma \ref{less n} tells us that  $y\leq x<n$ when $x-y\leq n-2$.  Hence we only need to focus on the case:
\begin{center}
$x-\frac{2}{n}x\geq x-y\geq n-1$.
\end{center} 
When $n=3$, we get $x-\frac{2}{n} x= x-\frac{2}{3} x=\frac{1}{3} x \geq n-1=3-1=2 $,  that is, $x\geq 6$. However, we already know that $x<2n=2\times 3=6$ by \cite[Lemma 3.6]{jie}. Hence when $n=3$, we always have $x<3$ (or by \cite[Section 1. Theorem]{Claus2}). When $n\geq 4$, we have $x-\frac{2}{n} x \geq n-1$, that is,  $\frac{2}{n-2}+n+1\leq x\leq 2n-1$ and  $(\frac{2}{n-2}+n+1)\times \frac{2}{n}=2+\frac{2}{n-2} \leq \frac{2}{n}x\leq  y\leq x- n+1\leq 2n-1-n+1=n$.  Then $ 3\leq y< n $ and $ 3+n-1=n+2  \leq y+n-1 \leq x \leq 2n-1$ by Corollary \ref{y}.

Suppose that there exists  an elementary module  $M\in \modd \cK_n$  with $\dimu M=(x,y)$, where $x-y\geq n-1$ and $n\geq 4$. Let $0\neq m\in M_1$, and let $U=\Lambda_n.m$.  We assume  that $\dimu U=(1,y')$. If $y=y'$ for any $m$, then we have preinjective dimension vector $\dimu M/U=(x-1,0)$. Using the embedding $\iota$, it is easy to see that $M$ is also an elementary module for $\cK_{n+1}$. Furthermore,  it is also an elementary module for the algebra $\cK_{x-y+2}$ by the map $\iota$, where $x-y+2\geq n-1+2=n+1>n$.  However, we have $x\leq y+(x-y+2)-2$,  whence $x<x-y+2\leq x-3+2=x-1$ by Lemma \ref{less n}, this is a contradiction. Hence there exists some $0\neq m'\in M_1$ such that $U=\Lambda_n.m'$ with $\dimu U=(x,y')$ and $y'\leq y-1$. 

 If $y'\leq y-2$ for some $m$,   then we can reconstruct a new submodule $U=\Lambda_n.m\oplus S(2)^{\oplus p}$  with $\dimu U=(1,y-2)$, that is,  $\dimu M/U=(x-1,2)$, where $p\in \mathbb{N}_0$. On the other hand,  we have $n+1\leq x-1\leq 2n-2$. Let $t=x-1$. We have

\begin{equation}
\begin{split}
q(x-1,2)& =(x-1)^2+4 -2n(x-1)\\&=  t^2-2nt+4 \\&= (t-n)^2+4-n^2 \\& \leq (2n-2-n)^2+4-n^2\\&= 8-4n\\ & <0.
\end{split}
\end{equation}
However, $q(1,y-2)=1+(y-2)^2-n(y-2)=1+(y-n-2)(y-2)<0$, $3\leq y< n$. Hence the module $M$ is not elementary by Lemma \ref{regular}.     If   $y'=y-1$ for any $m$,  then $\dimu M/U=(x-1,1)$. Now we get a  matrix

\begin{center}
$A_m=\begin{bmatrix}
\gamma_1.m\\
\vdots\\
\gamma_n.m
\end{bmatrix}.$
\end{center}
Then $\rk A_m\equiv y-1$ for any non-zero $m$. Moreover, we can take the matrix $A_m$ as a linear transformation from the vector space $k^y$ to the vector space $k^n$.  We define a vector space
\begin{center}
$\cV:=<A_m\mid 0\neq m\in M_1>$.
\end{center} 
Since $y\geq 3$, we have $l(y-1,y,n)\leq n-y+3\leq n$.  Hence $\dim_k \cV\leq n$.  Let $\{e_1,\cdots,e_{x}\}$ be a standard basis of $M_1$, that is,  the $i$-th coordinate is $1$ and all others are  $0$ in  each $e_i$. Since $x\geq n+2$,  the matrices in the set $\{A^t_{e_1},\cdots,A^t_{e_x}\}$ is not a basis of $L'\subseteq L(V^y, V^n)$. Then there exists $\alpha=(\alpha_1,\cdots,\alpha_x)\in k^{x}\setminus \{0\}$ such that
\begin{center}
$\rk (\sum^{x}_{i=1}\alpha_iA_{e_i})=\sum^x_{i=1}\begin{bmatrix}
\alpha_i\gamma_1.e_i\\
\alpha_i\gamma_2.e_i\\
\vdots\\
\alpha_i\gamma_n.e_i\\
\end{bmatrix} =\begin{bmatrix}\gamma_1.(\sum^x_{i=1}\alpha_ie_i) \\
\gamma_2.(\sum^x_{i=1}\alpha_ie_i) \\
\vdots \\
\gamma_n.(\sum^x_{i=1}\alpha_ie_i) \\
\end{bmatrix}=\begin{bmatrix}
\gamma_1.m'\\
\gamma_2.m'\\
\vdots\\
\gamma_n.m'
\end{bmatrix}<y-1,$ where $m'=\sum^x_{i=1}\alpha_ie_i$.
\end{center}
That is to say,  there exists an element $m'\in M_1$ such that $\dimu \Lambda_n.m'=(1,w)$ with $w\leq y-2$, a contradiction.

Finally, we can see that there exist at least two different elements   $m', m''\in M_1$ such that $ \dimu \Lambda_n.m'=(1,y-1), \hspace{0.1cm} \dimu \Lambda_n.m'=(1,y)$, and there always have $\dimu \Lambda_n.m =(1, w)$ with $w\geq  y-1$ for any $0\neq m\in M_1$.

\end{proof}

We now give an example. 

\begin{example}
We now give a module $M\in \modd\cK_4$ with $\dimu M=(3,6)$, 
 \[
\begin{tikzcd}
   k^3
    \arrow[r, draw=none, "{\vdots}" description]
    \arrow[r, bend left,        "M(\gamma_1)"]
    \arrow[r, bend right, swap, "M(\gamma_4)"]
    &
    k^6,
\end{tikzcd}
\]

\end{example}
where 

\[ M(\gamma_1)=\begin{bmatrix}
1 & 0 & 0 \\
0 & 1 & 0\\
0 & 0 & 1\\
0 & 0 & 0\\
0 & 0&0\\
0 & 0&0
\end{bmatrix}, \quad M(\gamma_1)=\begin{bmatrix}
0 & 0 & 0 \\
1 & 0 & 0\\
0 & 1 & 0\\
0 & 0 & 1\\
0 & 0&0\\
0 & 0&0
\end{bmatrix}, \quad M(\gamma_3)=\begin{bmatrix}
0 & 0 & 0 \\
0 & 0 & 0\\
1 & 0 & 0\\
0 & 1 & 0\\
0 & 0&1\\
0 & 0&0
\end{bmatrix}, \quad M(\gamma_4)=\begin{bmatrix}
0 & 0 & 0 \\
0 & 0 & 0\\
0 & 0 & 0\\
1 & 0 & 0\\
0 & 1&0\\
0 & 0&1
\end{bmatrix}. \]
Let  $0\neq u=(a,b,c)\in M_1$. Then we have a matrix
\[A_u=\begin{bmatrix}
\gamma_1.u\\
\gamma_2.u\\
\gamma_3.u\\
\gamma_4.u
\end{bmatrix} =  \begin{bmatrix}
a & b & c & 0 & 0 & 0\\
0 &  a& b & c & 0 & 0 \\
0 & 0 & a & b & c & 0\\
 0 & 0 &0 &a & b & c
\end{bmatrix}. \]
We know that $\rk A_u=4$, and $M$ is indecomposable and regular. Unfortunately, module $M$ is not elementary. If we let $u=(1,0,0)$, $v=(0,1,0)$, then we can get

\[\rk A_{u, v}=\rk \begin{bmatrix}
\gamma_1.u\\
\gamma_2.u\\
\gamma_3.u\\
\gamma_4.u\\
\gamma_1.v\\
\gamma_2.v\\
\gamma_3.v\\
\gamma_4.v\
\end{bmatrix} = \rk  \begin{bmatrix}
1 & 0 & 0 & 0 & 0 & 0\\
0 &  1& 0 & 0 & 0 & 0 \\
0 & 0 & 1 & 0 & 0 & 0\\
 0 & 0 &0 &1 & 0 & 0\\
 0& 1 & 0 & 0 & 0 & 0\\
0 &  0& 1 & 0 & 0 & 0 \\
0 & 0 & 0 & 1 & 0 & 0\\
 0 & 0 &0 &0 & 1 & 0
\end{bmatrix}=5. \]

It says that the module $U$ generated by the elements $u,v$  is regular with $\dimu U=(2,5)$. However, $\dimu M/U=(1,1)$ is still regular. According to Lemma \ref{regular}, $M$ is not elementary.

 Suppose that $M\in \modd \cK_n$ is an elementary module with dimension vector  $(x,y)\in \mathbf{F}$, where $x<n$.  According to Lemma \ref{n-1}, we know that  $M$ can be seen as  a representation of the quiver

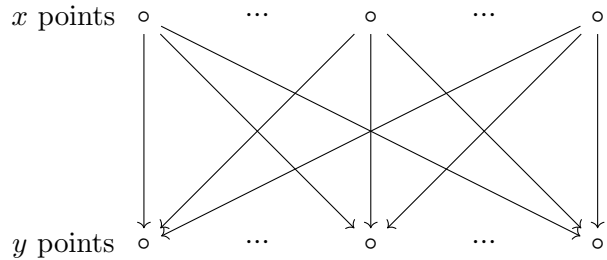
\begin{center}
\begin{tikzpicture}
\node (00) at (0,0) {$\circ$};
\node (08) at (-1,0) {$x \text{ points }$};
\node (25) at (1.5,0) {$\cdots$};

\node (30) at (3,0) {$\circ$};
\node (60) at (6,0) {$\circ$};

\node (35) at (4.5,0) {$\cdots$};

\node (0-3) at (0,-3) {$\circ$};
\node (0-9) at (-1,-3) {$y\text{ points }$};
\node (25) at (1.5,-3) {$\cdots$};
\node (3-3) at (3,-3) {$\circ$};

\node (35) at (4.5,-3) {$\cdots$};
\node (6-3) at (6,-3) {$\circ$};

\path [->] (00) edge (0-3)
           (00) edge(3-3)
           (00) edge(6-3)
           ;
\path[->]          (30) edge (0-3)      
         (30) edge  (3-3)                    
           (30) edge  (6-3)
                   ;
             
\path[->] (60) edge (0-3)  
   (60) edge (3-3)             
   (60) edge (6-3) ;                        
\end{tikzpicture}
\end{center}
where each  point of upper row has $y$  arrows. For the elementary module $M$, we just put the $1$-dimensional vector space $k$ at each point. We will give a precise construction later. Let $m\in M_1$. We define a matrix $A_m=\begin{bmatrix}
\gamma_1.m & \cdots & \gamma_n.m
\end{bmatrix}^t$.  If $\{m_1,\cdots,m_x\}$ is a basis of $M_1$, then we get $A_{m_i}=\begin{bmatrix}
\gamma_1.m_i & \cdots & \gamma_n.m_i
\end{bmatrix}^t, 1\leq i\leq n$, and  we have  $\rk A_{m_i}=y$  by Lemma \ref{n-1}. Let $\cH_M=k<A_{m_i}\mid  1\leq i\leq x>.$   We have

\begin{Lemma}\label{dimen}\cite[Lemma 3.10, Lemma 3.12]{jie}
Let $(x,y)\in \mathbf{F}$ with $x<n$.  Suppose that $M\in\modd \cK_n$ is an elementary module with $\dimu M=(x,y)$.  Then $\dimu \Lambda_n.m=(1,y)$ for any $0\neq m\in M_1$ and $\dim_k \cH_M=x$.

\end{Lemma}

In fact, we have the following corollary.
\begin{corollary}\label{equal}
Let $A\in \cH_M$ be a non-zero element. Then we have $\rk A=y$. 
\end{corollary}
\begin{proof}
Suppose that there exist some non-zero $b_i\in k$ such that  $A=\sum^x_{i=1}b_i A_i\in \cH_M$ with $\rk A<y$ .  It says that 
\begin{center}
$\rk (\sum^x_{i=1}b_iA_{m_i})=\rk \begin{bmatrix}
\gamma_1.(\sum^x_{i=1}b_im_i) \\
\vdots \\
\gamma_x.(\sum^x_{i=1}b_im_i) \\
\vdots\\
 \gamma_n.(\sum^x_{i=1}b_im_i)
\end{bmatrix}<y,  $
\end{center}
Hence we can find an element $0\neq m'=\sum^x_{i=1}b_im_i\in M_1$ such that $\rk A_{m'}<y$. However, the submodule $\Lambda_n.m'$ of $M$ has dimension vector $\dimu \Lambda_n.m'=(1, \rk A_{m'})$. According to Lemma \ref{dimen}, this is a contradiction.

\end{proof}
 
Lemma \ref{dimen} and Corollary \ref{equal} tell us that  $\cH_M$ is a linearly subspace of fixed rank in $L(V^y, V^n)$ with $\dim_k\cH_M=x$ when $M$ is elementary. Let  $(x,y)\in \mathbf{F}$ with $x<n$. We now construct a  module $X=X(x,y;x+y-1)=(X_1,X_2,X(\gamma_i)_{1\leq i\leq x+y-1})$ of $\cK_{x+y-1}$: let $\{e_1,\cdots,e_x\}$ be a standard basis of $X_1$, and  let $\{e'_1,\cdots,e'_y\}$ be a standard basis of $X_2$, correspondingly.  We use $[s;a_1,a_2,\cdots,a_y]$ to denote  the arrow $\gamma_{a_i}$ mapping $e_s$ to $e'_{i}$ in $X(x,y;x+y-1)$, and call it the \textit{arrow basis} of $e_s$, where $1\leq s\leq x, 1\leq  i \leq y, 1\leq a_i \leq x+y-1$. For $e_1$, we define its arrow basis being $[1;1,2,\cdots,y]$. For the second $e_2$, we start from $2$ to  $y+1$, that is, the arrow basis of $e_2$ is $[2;2,3,\cdots,y+1]$. For $e_3$,  we repeat this process by starting from $3$ to   $y+2$.  We keep doing it. Finally, we can get an arrow basis of $X(x,y;x+y-1):[1;1,2,\cdots,y], [2;2,3, \cdots,y+1], \cdots, [s;s ,s+1,\cdots,y+s-1],\cdots,  [x;x,x+1,\cdots,x+y-1],1\leq s\leq x.$
That is to say, 

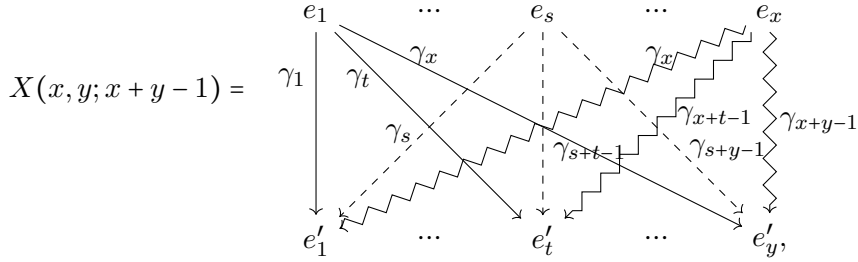
\begin{center}

\begin{tikzpicture}
\node (00) at (0,0) {$e_1$};
\node (25) at (1.5,0) {$\cdots$};
\node (30) at (3,0) {$e_s$};
\node (60) at (6,0) {$e_x$};

\node (35) at (4.5,0) {$\cdots$};
\node (22)  at (-2.5,-1){$X(x,y;x+y-1)=$};

\node (0-3) at (0,-3) {$e'_1$};
\node (25) at (1.5,-3) {$\cdots$};
\node (3-3) at (3,-3) {$e'_t$};
\node (35) at (4.5,-3) {$\cdots$};
\node (6-3) at (6,-3) {$e'_{y},$};

\path [->] (00) edge node[pos=.25,left]{$\gamma_1$} (0-3)
           (00) edge node[pos=.25, left]{$\gamma_t$} (3-3)
           (00) edge node[pos=.15, right]{$\gamma_{x}$}(6-3)
           ;
\path[dashed, ->]          (30) edge node[pos=.55, left]{$\gamma_{s}$} (0-3)      
         (30) edge node[pos=.65,right]{$\gamma_{s+t-1}$} (3-3)                    
           (30) edge node[pos=.65,right]{$\gamma_{s+y-1}$} (6-3)
                   ;
             
\draw[snake,segment length=10pt,->] (60)--(0-3)  node[pos=.15,left]{$\gamma_x$}; 
  \draw[snake,segment length=10pt,->] (60)--(3-3)  node[pos=.45,right]{$\gamma_{x+t-1}$};             
   \draw[snake,segment length=10pt,->] (60)--(6-3)  node[pos=.5,right]{$\gamma_{x+y-1}$};                        
\end{tikzpicture}

\end{center}
Let $e'_j=0$ when $j\nin \{1\,\cdots,y\}$. Then  $\gamma_i.e_j=e'_{i-j+1},i\in \{1,2,\cdots,x+y-1\}$, $j\in \{1,\cdots,x\}$. Now we give an example. 

\begin{example}
Let $X(4,4;7)\in\modd \cK_7$. Then we can get its arrow basis:
\begin{center}
$[1;1,2,3,4],[2;2,3,4,5],[3;3,4,5,6], [4;4,5,6,7]$.
\end{center}
The structure of  $X(4,4;7)$ would be the following
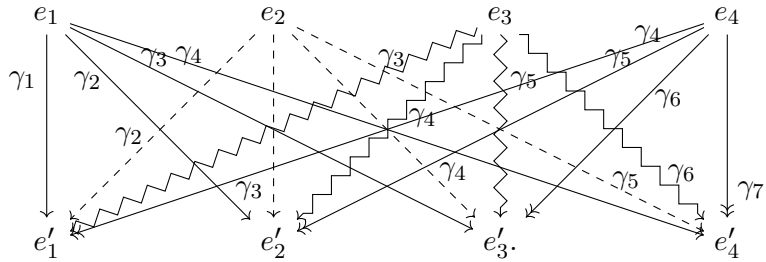
\begin{center}

\begin{tikzpicture}
\node (00) at (0,0) {$e_1$};

\node (30) at (3,0) {$e_2$};
\node (60) at (6,0) {$e_3$};
\node (90) at (9,0) {$e_4$};

\node (0-3) at (0,-3) {$e'_1$};

\node (3-3) at (3,-3) {$e'_2$};

\node (6-3) at (6,-3) {$e'_{3}.$};
\node (9-3) at (9,-3) {$e'_4$};

\path [->] (00) edge node[pos=.25,left]{$\gamma_1$} (0-3)
           (00) edge node[pos=.25, left]{$\gamma_2$} (3-3)
           (00) edge node[pos=.15, right]{$\gamma_{3}$}(6-3)
           (00) edge node[pos=.15, right]{$\gamma_{4}$}(9-3)
           ;
\path[dashed, ->]          (30) edge node[pos=.55, left]{$\gamma_{2}$} (0-3)      
         (30) edge node[pos=.85,left]{$\gamma_{3}$} (3-3)                    
           (30) edge node[pos=.75,right]{$\gamma_{4}$} (6-3)
           (30) edge node[pos=.75,right]{$\gamma_{5}$} (9-3)
                   ;
             
\draw[snake,segment length=10pt,->] (60)--(0-3)  node[pos=.15,left]{$\gamma_3$}; 
  \draw[snake,segment length=10pt,->] (60)--(3-3)  node[pos=.45,right]{$\gamma_{4}$};             
   \draw[snake,segment length=10pt,->] (60)--(6-3)  node[pos=.25,right]{$\gamma_{5}$}
   ; 
   \draw[snake,segment length=10pt,->] (60)--(9-3)  node[pos=.75,right]{$\gamma_{6}$}
   ;     
 \draw[->>] (90)   edge node[pos=.05,left]{$\gamma_4$} (0-3)
   (90)   edge node[pos=.15,left]{$\gamma_5$} (3-3)
   (90)   edge node[pos=.35,right]{$\gamma_6$} (6-3)
   (90)   edge node[pos=.85,right]{$\gamma_7$} (9-3);
                                           
\end{tikzpicture}

\end{center}
Then
\begin{center}

 $X(4,4;7)=\begin{tikzcd}
   k^4 
    \arrow[r, draw=none, "{\vdots}" description]
    \arrow[r, bend left,        "X(\gamma_1)"]
    \arrow[r, bend right, swap, "X(\gamma_7)"]
    &
     k^4,
\end{tikzcd}$
\end{center}
where $X(\gamma_1)=\begin{bmatrix}
1 & 0 & 0 &0\\
0 & 0 & 0 & 0\\
0 & 0 & 0 &0 \\
0 & 0 & 0 &0 
\end{bmatrix}, X(\gamma_2)=\begin{bmatrix}
0 & 1& 0 & 0\\
1 & 0 & 0 & 0\\
0 & 0 & 0 & 0 \\
0 & 0 & 0 & 0
\end{bmatrix}$, $X(\gamma_3)= \begin{bmatrix}
0 & 0& 1 & 0\\
0 & 1& 0 & 0\\
1& 0 & 0 & 0 \\
0 & 0 & 0 & 0
\end{bmatrix}$, $X(\gamma_4)= \begin{bmatrix}
0 & 0& 0 & 1  \\
0 & 0& 1 & 0 \\
0& 1 & 0 & 0\\
1 & 0 & 0 & 0
\end{bmatrix}$,

$X(\gamma_5)= \begin{bmatrix}
0 & 0& 0 & 0\\
0 & 0& 0  & 1 \\
0& 0 & 1 & 0 \\
0 & 1 & 0 & 0
\end{bmatrix}$,  $ X(\gamma_6)= \begin{bmatrix}
1 & 0 & 0 &0\\
0 & 0 & 0 & 0\\
0 & 0 & 0 &1 \\
0 & 0 & 1 &0 
\end{bmatrix}$, $ X(\gamma_7)= \begin{bmatrix}
0 & 0 & 0 &0\\
0 & 0 & 0 & 0\\
0 & 0 & 0 &0 \\
0 & 0 & 0 & 1 
\end{bmatrix}$.

\end{example}

In fact, we have the following

\begin{Lemma}\label{theorem}
Let   $(x,y)\in \mathbf{F}$ with $x<n$. Then the module $X(x,y;x+y-1)$ is elementary.

\end{Lemma}
\begin{proof}
The proof is the same with that of \cite[Lemma 4.1]{jie}.
\end{proof}

Note that $x=x+y-1-y+1\leq l(y,y,x+y-1)\leq x+y-1+y-2y+1=x$. Let $\cH\in \cL(x,y)$. Then we have $\dim_k \cH=x$.  Let $\{A_i=(a^i_{pq})_{(x+y-1)\times y}\mid 1\leq i\leq x\}$ be a  basis of linear subspace $\cH$, and let $\{d_1,\cdots,d_x\}$ be a  basis of $k^x$, where $d_i=(d^i_1, d^i_2,\cdots,d^i_x)$, $1\leq i\leq x$.  Suppose that we have a module $M_{\cH}\in \modd\cK_{x+y-1}$ with the action of $\gamma_j$ satisfying the following condition
\[ A_{i}=(a^i_{pq})_{(x+y-1)\times y}=\begin{bmatrix}
\gamma_1.d_i\\\
\gamma_2.d_i\\
\vdots\\
\gamma_{x+y-1}.d_i
\end{bmatrix}, \quad 1\leq i\leq x, \]
where $\gamma_i=(c^p_{qw})_{y\times x}$, $1\leq p \leq x+y-1$, $1\leq q\leq y$, $1\leq w\leq x$. Then we have

\[\begin{cases}
\sum^x_{w=1}c^p_{qw}d^1_w=a^1_{pq} \\
\sum^x_{w=1}c^p_{qw}d^2_w=a^2_{pq}\\
   \quad \vdots\\
\sum^x_{w=1}c^p_{qw}d^i_w=a^i_{pq} \\
  \quad  \vdots\\
\sum^x_{w=1}c^{p}_{qw}d^x_w=a^{x}_{pq} 
\end{cases}\]

Hence we get

\[ \begin{bmatrix}
c^p_{q1}\\[5pt]

c^{p}_{q2}\\[5pt]

\vdots\\[5pt]

c^p_{qx}
\end{bmatrix}= D^{-1}\begin{bmatrix}
a^1_{pq}\\[5pt]
a^{2}_{pq}\\[5pt]
\vdots\\[5pt]
a^x_{pq}
\end{bmatrix} \]

 where

\[ D=\begin{bmatrix}
d^1_1 & d^1_2 & \cdots & d^1_x\\
d^2_1 & d^2_2 & \cdots & d^2_x\\
\vdots & \vdots & \ddots & \vdots \\
d^x_1 & d^x_2 & \cdots & d^x_x
\end{bmatrix}.  \]
Then we have

\[ \gamma_i=\begin{bmatrix}
a^1_{i1} & a^2_{i1} & \cdots & a^x_{i1}\\
a^1_{i2} & a^2_{i2} & \cdots & a^x_{i2}\\
\vdots & \vdots & \ddots & \vdots\\
a^1_{ix} & a^2{ix} & \cdots & a^x_{ix}\\
\end{bmatrix} \cdot (D^{-1})^t=B_i (D^{-1})^t, \]
where $B_i=(a^p_{iq})_{x\times x}$, $1\leq i\leq x+y-1$. Moreover, one can see that if we choose another basis $\{f_i\mid f_i=(f^i_1,\cdots,f^i_x), 1\leq i\leq x\}$, then we can similarly get $\hat{\gamma_i}=B_i (F^{-1})^t$. Clearly, we have the isomorphism 

\[ \begin{tikzcd}[sep=2.4cm]
  k^x 
  \arrow[d, "F^t(D^{-1})^t"]
  \arrow[r, draw=none, "{\vdots}" description]
    \arrow[r, bend left,        "\gamma_1"]
    \arrow[r, bend right, swap, "\gamma_{x+y-1}"]
    &
    k^y \arrow[d, "E_y"] \\   
    k^x
  \arrow[r, draw=none, "{\vdots}" description]
    \arrow[r, bend left,        "\hat{\gamma}_1"]
    \arrow[r, bend right, swap, "\hat{\gamma}_{x+y-1}"]
    &
    k^y 
\end{tikzcd} \]

where $E_y$ is the $y\times y$ identity matrix and 

\[F=\begin{bmatrix}
f^1_1 & f^1_2 & \cdots & f^1_x\\
f^2_1 & f^2_2 & \cdots & f^2_x\\
\vdots & \vdots & \ddots & \vdots \\
f^x_1 & f^x_2 & \cdots & f^x_x
\end{bmatrix}. \]

 Let $0\neq m=(m_1,m_2,\cdots,m_x)\in k^x$. Then we have 
 
\[ \rk \begin{bmatrix}
\gamma_1.m\\
\gamma_2.m\\
\vdots\\
\gamma_{x+y-1}.m
\end{bmatrix}=\rk ( \begin{bmatrix}
\gamma_1.(\sum^x_{i=1} m_ie_i)\\
\gamma_2.(\sum^x_{i=1} m_ie_i)\\
\vdots\\
\gamma_{x+y-1}.(\sum^x_{i=1} m_ie_i)
\end{bmatrix}) =  \rk ( \sum^x_{i=1} m_i\begin{bmatrix}
\gamma_1.e_i\\
\gamma_2.e_i\\
\vdots\\
\gamma_{x+y-1}.e_i
\end{bmatrix})=\rk (\sum^x_{i=1}m_iA_i)=y,\]
We know that $M_{\cH}$ is elementary by \cite[Lemma 3.10]{jie}.  We define two sets

\begin{enumerate}
\item $\cM(x,y)=\{M\in \modd\cK_{x+y-1}\mid M \text{ is elemenetary with } \dimu M=(x,y)\}$,

\item $\cL(x,y)=\{\cH \subseteq L(V^{y}, V^{x+y-1})\mid \cH  \text{ is a maximum subspace with } \rk v=y, \forall  0\neq v\in \cH\}$.
\end{enumerate}

Then we have the following
\begin{Theorem}
Let   $(x,y)\in \mathbf{F}$ with $x<n$. Then there exists a bijection map between the set $\cM(x,y)$ and the set $\cL(x,y)$ under the isomorphisms.

\end{Theorem}
\begin{proof}
Clearly, we have two maps

\[ \phi: \cM(x,y)\rightarrow \cL(x,y); M\mapsto \cH_M,  \]

\[ \varphi: \cL(x,y)\rightarrow \cM(x,y); \cH\mapsto M_{\cH}.\]

According to our above discussion,  we have

\[(\varphi\phi)(M)=\varphi(\cH_M)=M_{\cH_M}\cong M,\]

\[(\phi\varphi)(\cH)=\phi(M_{\cH})=\cH_{M_{\cH}}\cong \cH.\].

\end{proof}

\clearpage

\end{document}